\documentclass[12pt]{amsart}
\usepackage{amssymb}
\usepackage{amsmath}
\usepackage{amsthm, fullpage}
\usepackage{stmaryrd} 
\usepackage{mathtools}
\usepackage{soul}
\usepackage{caption, color}
\usepackage{subcaption}
\usepackage{float}
\newtheorem{remark}{Remark}

\usepackage{graphicx}
\usepackage{tabulary}
 \usepackage{hyperref}
\usepackage{amsmath}
\usepackage{tikz}
\usepackage{tikz-cd}
\usepackage{booktabs}
\usetikzlibrary{shapes.geometric}
\usepackage{amscd}
\usepackage{mathabx}
\numberwithin{equation}{section}
\usepackage[original]{imakeidx}
\makeindex 
\usepackage{mathrsfs}
\usepackage{graphics}
\usepackage{eucal}
\usepackage{mathrsfs}
\usepackage{mathtools}

\newenvironment{nouppercase}{%
  \renewcommand{\uppercasenonmath}[1]{}}{}

\newtheorem{theorem}{Theorem}[section]
\newtheorem{proposition}{Proposition}[section]
\newtheorem{lemma}[theorem]{Lemma}
\newtheorem{corollary}[theorem]{Corollary}
\newtheorem{claim}[theorem]{Claim}
\newtheorem{sublemma}[theorem]{Sublemma}
\newtheorem{definition}[theorem]{Definition}
\newtheorem{qn}[theorem]{Question}

\theoremstyle{remark}
\title{On the Piecewise Linear Perturbations of the Doubling Map}

\begin{document}

\author{Anubrato Bhattacharyya}
\address{Department of Mathematics, Presidency University, 86/1, College Street, Kolkata -
700073, West Bengal, India.}
\email{anubrato02@gmail.com}

\thanks{The research of first author was funded by UGC [NTA Ref. No. 201610319430], Govt. of India. The second author was supported  partly by the Department of Science and Technology (DST), Govt. of India, under the Scheme DST FIST [File No. SR/FST/MS-I/2019/41] and NBHM travel grant 0207/15(3)/2023/R\&D-II/11964.}

\author{Kuntal Banerjee}
\address{Department of Mathematics, Presidency University, 86/1, College Street, Kolkata -
700073, West Bengal, India.}

\email{kbanerjee.maths@presiuniv.ac.in}

\thanks{Corresponding author: First author}

\keywords
{circle maps\\2020 Mathematics Subject Classification: 37E10, 37D05, 37B55, 37G15, 37G35, 37C70}

\begin{abstract}
Inspired by the 2007 work by  M.~Misiurewicz and A.~Rodrigues \cite{MR}, we consider a family of circle maps that are perturbations of the doubling map on the circle by a piecewise linear map.  We call this the  \textit{piecewise linear perturbation of the doubling map} (PLPDM) and it is given by the formula,
        $f_{a,b}(x)= \displaystyle \bigl(2x+a+\dfrac{b}{2} S(x) \bigr) \sslash 1 \quad {\text {for }} x, a, b \in [0,1] $, where  $y \sslash 1$ means $y \mod 1$ (or simply, the fractional part of $y$) and $S(x)$ is the piecewise linear approximation of $\sin 2\pi(x-1/4)$. The map $S(x)$ is called the straight sine map. Define the hyperbolic set, $\mathcal{H}= \{ (a,b) \in \mathbb{R}/\mathbb{Z} \times [0,1] : f_{a,b} \text{ has an attracting cycle} \}$. Tongues are defined as the components of $\mathcal{H}$
 that touch the ceiling $\{b=1\}$ in a non degenerate interval. Any other component is referred to as an Eye. We show the uniqueness of the attracting cycle of $f_{a,b}$ for  $(a,b) \in \mathcal{H}$. We then define \textit{type} and prove the existence of the tongues of all types. We also show how combinatorics of the attracting orbit determines if the component is a tongue or an eye. We show that $f_{a,b}$ is conjugate to the doubling map if $(a,b) \notin \overline{\mathcal{H}}$. Some experimental proof of the existence of eyes in the parameter space corresponding to different combinatorics will be shown.
\end{abstract}

\begin{nouppercase}
\maketitle
\end{nouppercase}

\section{Introduction}

In 2007,  M.~Misiurewicz and A.~Rodrigues studied the following family $D_{a,b}:\mathbb T \to \mathbb T$ of maps on circle $\mathbb R/\mathbb Z=\mathbb {T}$. These are called the Double Standard Maps (DSM) in their work (\cite{MR}) and they are defined as follows,
   \begin{align*}
    D_{a,b}(x)=\left(2x+a+\dfrac{b}{\pi} \sin(2\pi x)\right)\sslash 1, \quad {\text {for }} x\in \mathbb T=\mathbb R/\mathbb Z, a \in \mathbb{R}, 0 \leq b \leq 1
\end{align*}
Here $y\sslash1$ means $y \mod 1$.  

Recall the doubling map on the circle
\begin{align*}
    D(x)= 2x\sslash1 \quad {\text {for }} x\in \mathbb T.
\end{align*}
The maps in DSM family can be seen as analytic perturbations of $D$, parameterized by $a,b$ where the parameter space is,  $\mathcal P_{\textrm{DSM}}=\{(a,b):a\in \mathbb R, 0\le b\le 1\}$.

The map $D_{a,b}$ is also  semiconjugate to $D$, for any $(a,b)\in \mathcal P$. Misiurewicz and  Rodrigues prove the existence of atmost one attracting cycle for any  parameter $(a,b) \in \mathcal{P}_{\text{DSM}}$ (Theorem 3.5, \cite{MR}). The maps in DSM family can also be seen as a hybrid between the family of Standard Maps and expanding circle maps. The maps in the Standard Family  have the same formula but with out the $2$ multiplied, as written below
\begin{align*}
    A_{a,b}(x)=\left(x+a+\dfrac{b}{2\pi} \sin(2\pi x)\right)\sslash 1, \quad{\text {for }} x\in \mathbb T, a \in \mathbb{R} ,0 \leq b < 1.
\end{align*}

Tongues in DSM family are subsets of the parameter space such that the corresponding map for any parameter value within a tongue has an attracting periodic cycle. 
The fact, that there  cannot be more than one attracting cycle for maps in the DSM family, follows from studying the complexified map. This complexification $g_{a,b}$ of $D_{a,b}$ extends naturally to a holomorphic self map of the punctured complex plane as written presently
    \begin{align*}
    g_{a,b}: \mathbb{C}^* \to \mathbb{C}^*,\quad  g_{a,b}(z)=e^{2\pi i a}z^2\exp\left( bz-\dfrac{b}{z}  \right).
\end{align*} 

The uniqueness of the attracting cycle for maps in  DSM family is fundamental to the work of the authors in \cite{MR}. The proof of uniqueness  crucially  uses the holomorphic extension of the complexification. Now if one considers a perturbation of the doubling map that is not analytic unlike in DSM, the complexification which is defined on unit  circle may not extend holomorphically to the punctured complex plane. One may then ask if there is at most one attracting cycle when the perturbation is non analytic. This curiosity can be considered as the starting point of our paper. The perturbation that we consider here is probably the simplest non-smooth example possible, namely a piecewise linear approximation of $\sin 2\pi(x-1/4)$. We call this family as the  \textit{piecewise linear perturbations of the doubling map} (PLPDM) as is defined  below. We will often identify $\mathbb{T}$ with $I=[0,1]/_{0\sim1}$  and use interchangeably.
        \[
    f_{a,b}(x)= \left(2x+a+\dfrac{b}{2} S(x) \right)\sslash1 \quad {\text {for }} x\in I \] where $S$ is called the straight sine map and $S$ is defined in the following manner,
\[S(x)= \begin{cases}
              4x-1 &\text{  if $x \in [0,\frac{1}{2}]$,}\\
              -4x+3 & \text{  if $x \in [\frac{1}{2},1)$.}\end{cases}\] 

\begin{figure}
    \centering
    \includegraphics[scale=.2]{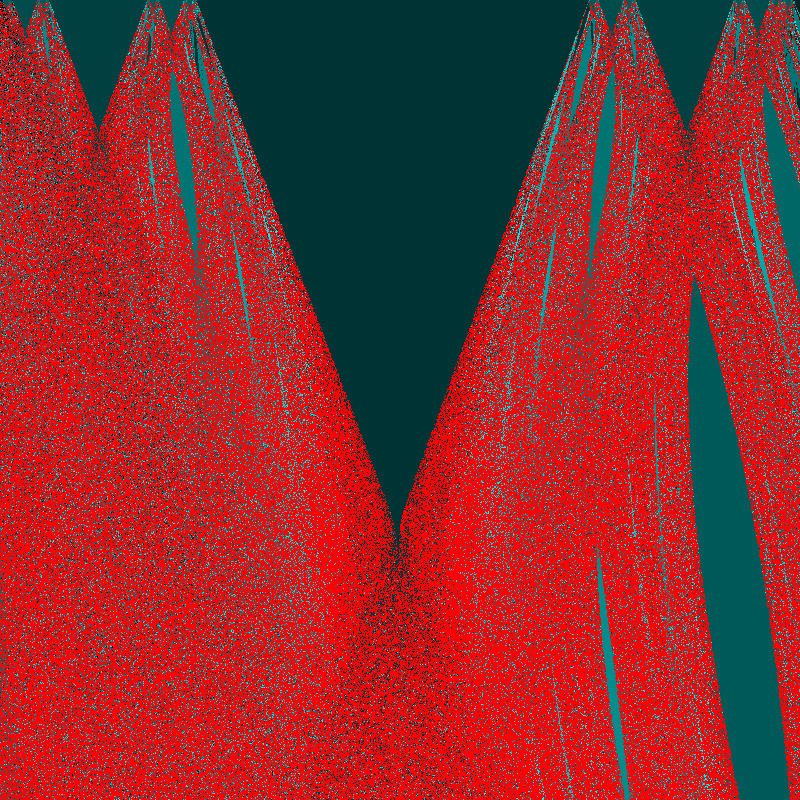}
    \caption{\small Tongues and Eyes in PLPDM (By A.~Dezotti)}
    
\end{figure}

The parameter space for this family is $\mathcal P=\{(a,b):a\in \mathbb R/\mathbb{Z}, 0\le b\le 1\}$.

Every map in this family is a degree two piecewise linear map which is expanding on $(0,1/2)$ and contracting on $(1/2,1)$. Moreover, each map has two points of non differentiability at 0 and 1/2. These together pose a challenge in studying the dynamical behaviour of the maps in this family. 

Upon experimentation with PLPDM family, a new phenomenon becomes apparent in the parameter space. Apart from tongues, we observe new kind of shapes in the parameter space of PLPDM, which we are going to refer to as ``eyes". These shapes are absent for both the standard family of circle maps and the DSM family.  We will prove   in section \ref{eye} the following basic observation from the plotted images:
\begin{quote}
 \textit{Eye} is a connected  region that is a positive distance away from  $b=1$.  
\end{quote} 
\section{Preliminaries and main results}

First, we set up our notation. We denote, $I=[0,1], I_+= [ 0, \frac{1}{2} ]\text{ and, }I_-=[ \frac{1}{2},1].$ We may often use $\mathbb{R}/\mathbb{Z}=\mathbb{T}$ interchangeably with $I/\sim$ where, for $x,y \in I$, $x\sim y$ if and only if $x-y=1$. The length of an interval $J$ within the real line or the circle is denoted by $|J|$.

 By opening up the formula we have,
\begin{align*}
         f_{a,b}: \mathbb{R/Z} \to \mathbb{R/ Z} 
     \end{align*}
    \begin{align*}
      f_{a,b}(x)=(B^+x+A^+)\sslash1,x \in I_+,
\end{align*}
\begin{align*}
     f_{a,b}(x)=(B^-x+A^-)\sslash1,x \in I_-,
\end{align*}
\begin{align*}
    B^\pm = 2(1 \pm b),A^+ = (a-b/2),A^-=(a+3b/2).
\end{align*}

This means that for inputs $a,b,x \in \mathbb{R/Z}$, all the additions and multiplications are performed by thinking of $\mathbb{R/Z}$ as the  quotient ring.

\vspace{.25cm}

We define the itinerary of a point $x$ as a sequence of $+$ and $-$.  We say $x$ has itinerary,  $\sigma=(\sigma_1,\sigma_2,...,\sigma_n,...) $, where $\sigma_n=\pm $, accordingly as $f_{a,b}^n(x)\in I_{\pm}$. This sequence becomes periodic for a periodic point. 

\begin{definition}[Itinerary]
 The slope itinerary of an orbit $x_0, x_1, \ldots$  
 
 is the
 sequence $\sigma_j$ of + or - such that
 $x_j \in  I_{\sigma_j} $.  For a periodic orbit of period $p$ (for some $p \geq 1$), the slope itinerary is reduced to a
finite sequence $\sigma_0,\ldots,\sigma_{p-1}$. Such a sequence is called a periodic slope itinerary  (an attracting periodic orbit must have at least one $-$ in its itinerary). We may just use `itinerary' instead of `periodic slope itinerary' if there is no chance of ambiguity.  \\
 \indent
    Let $p \geq 1$, $\sigma \in \{-,+\}^p $. A parameter $(a,b)$ has slope itinerary
$\sigma$ if $f_{a,b}$ has a periodic orbit of period $p$ and periodic slope itinerary $\sigma$. If a periodic slope itinerary $\sigma$ has more than one minus then it is called non single minus periodic slope itinerary. We only consider itinerary of a point whose orbit avoids $0$ and $\frac{1}{2}$. 
\end{definition}

\begin{remark}
    If  $f_{a,b}$ has a periodic  cycle $x_0, x_1,\ldots, x_{p-1}$   of period $p$ with itinerary $\sigma=(\sigma_0,\ldots,\sigma_{p-1})$, then the slope of $f_{a,b}^p$ at $x_0$ is equal to $B^{-^{m}}B^{+{n}}$ where $m,n$ are numbers of $-$ and $+$, respectively in $\sigma$.
\end{remark}

\begin{definition}[Tongue and Eye]
     Let $\sigma$ be a periodic slope itinerary and let S be a connected component of the set of parameters corresponding to attracting cycles with periodic slope itinerary $\sigma$. The set S is called a
tongue if $\overline{S}$ intersects $\mathbb{T}\times \{1\}$ on a nontrivial interval. Otherwise $S$ is called an
eye.
\end{definition}

Notice that tongues and eyes are the connected components of the hyperbolic set $\mathcal{H}$ defined presently,
\[\mathcal{H}= \{ (a,b) \in \mathbb R/\mathbb{Z} \times [0,1) : f_{a,b} \text{ has an attracting cycle} \}.\] 

\begin{definition}[Neutral Cycle]\label{neutral cycle}
   Let $ g: \mathbb{R/Z} \to \mathbb{R/ Z} $. A periodic point $x$ of $g$ with period $n$ is called a neutral periodic point if $g^n$ is neither attracting nor repelling fixed point of $g^n$. The periodic orbit is called a neutral cycle or neutral periodic orbit.

\end{definition}

\subsection{Main Results}
\vspace{.3cm}
In this paper we have  tried to produce  a basic picture of the parameter space in the spirit of \cite{MR}. For instance, similar to the Double Standard Family we also show via Theorem \ref{uniqueness of periodic attractor}, the uniqueness of an attracting periodic cycle. Theorem \ref{imlicit} and its corollary tell us that the tongues and eyes are all open and so is the whole hyperbolic set. This combined with the uniqueness of periodic attractors let us define the notion of \textit{type} of a tongue, similar to  \cite{MR}. The following theorems \ref{tongue esist}, \ref{eye combinatorics}  demonstrate that the combinatorics of the periodic cycle in fact dictates if the hyperbolic component is a tongue or an eye.  A single minus itinerary corresponds to a tongue while the non single minus corresponds to an eye. Presence of the eyes  truly  separates the 
 parameter space in PLPDM from the case of  the  Double Standard Family. Theorem \ref{tongue esist} in fact tells us that tongues exist for all single minus itinerary while Theorem \ref{for all types} states that there are tongues for all \textit{types}. 

\begin{theorem}\label{tongue esist}
    Let $\sigma$ be a single minus periodic slope itinerary. Then the set of parameters $(a,b) \in \mathcal{P}$ with slope itinerary $\sigma$ is non empty. Further, any connected component of  the set of parameters with slope itinerary $\sigma$ is a tongue.
\end{theorem}

\begin{theorem}\label{eye combinatorics}
    Let $\sigma$ be a non single minus periodic slope itinerary of length $n$ and let $H$ be a non empty connected component of the set of parameters having periodic slope itinerary $\sigma$.
 
Then $H$ is an eye. 
\end{theorem}

Theorem \ref{conjugate doubling} shows that any parameter for which the corresponding map is without non repelling cycle is in fact conjugate to the Doubling map. We have used the definition of neutral periodic point from topological dynamics (see Definition \ref{neutral cycle}) because the maps in PLPDM are non-differentiable at $0 $ and $\frac{1}{2}$. Proposition \ref{break bif} posits that similar to the DSM family, the bifurcation in PLPDM family is also through the creation of a neutral cycle. The break points $0$ and $\frac{1}{2}$ themselves become neutral periodic points respectively at the right and left boundary of a hyperbolic component.

In the final section we pose two questions. The first one conjectures that any \textit{admissible post critical combinatorics} is realized  in the PLPDM family, creating a parallel with the parameter space of family of quadratic polynomials, i.e., the Mandelbrot set. The second question asks for a dynamically natural uniformization of hyperbolic components (tongues and eyes) that may lead to  analytic motion of the \textit{maximal chaotic set} inside a component. If true, this will establish a parallel with the DSM family (see \cite{BBM}).

\section{Uniqueness of  attracting cycle and definition of Type}

We wish to define something akin to `type' for tongues  in PLPDM family, as was done in DSM. In \cite{MR} the two key ingredients for the definition of type are the following.
\begin{itemize}
    \item Existence of semiconjugacy between maps in the Double Standard Family and the doubling map.
    \item  Existence of at most one attracting cycle for every parameter $(a,b) $ in the Double Standard Family.
\end{itemize}
  The semiconjugacy with the doubling map is available in the PLPDM case using the same machinary. Hence what we aim to prove is the existence of at most one attracting cycle for every parameter $(a,b) \in \mathcal P$.

Before that we state a lemma that the parameter space is symmetric with respect to $a=\frac{1}{4}$ where $a \in \mathbb{R}/\mathbb{Z}$. 
\begin{lemma} 
For $(a,b)\in \mathcal{P}$,
\[
    f_{1/2-a,b}(1/2-x) = 1/2 - f_{a,b}(x).
\]
\end{lemma}

 We will now prove an important lemma.
\begin{lemma}\label{0.5 in closure}
If $x$ is the attracting periodic point in some periodic orbit that is closest to $1/2$ within $I_-$, then $1/2$ is contained in the closure of the component of basin of attraction that contains $x$.
\end{lemma}
 \begin{proof}
By an interval $[x_1,x_2]$ or $(x_1, x_2)$ in $\mathbb{T}$ we mean the closed or open arc joining $x_1$ to $x_2$ in the counterclockwise direction. Let $n \in \mathbb{N}$ be the  period of $x$. Let the component of immediate basin of attraction containing $x$ be $B_x$ and let $\overline{B_x}=[x_0,y_0]$. Then it is clear that $x_0,y_0$ are fixed points of $f^n$. Indeed the point $x_0$ is a fixed point of $f^n$ that is repelling from the right as $x$ is an attracting fixed point of $f^n$. Observe that, $f^i(B_x)=B_{f^i(x)}$ where $B_x$ and $B_{f^i(x)}$ are the basin components of $x$ and $f^i(x)$ respectively. 
If $1/2$ is not contained in the closure of $B_x$ then there is some $x_0 \in (1/2,x)$ which is  fixed point of $f^n$ that is repelling from the right. Now
$f^i(x) \in I_{-} \implies f^i(x_0) \in I_{-}$, as $f$ is orientation preserving and the components of basin of attraction are disjoint. 
This means that for any $y \in (x_0,x)$  whose forward orbit avoids the break points $0,\frac{1}{2}$ till the $n$'th iterate, the number of $'-'$ in the itinerary of $y$ is greater than or equal to that of $x$. Hence there exists some $\delta>0$ such that the slope of $f^n$ on $(x_0,x_0+\delta)$ is less than or equal to the slope of $f^n$ at $x$, which is a contradiction as $x_0$ is  repelling from the right (since $x$ is attracting). Hence our assumption was wrong, meaning that $1/2$ is in the closure of $B_x$. 
    \end{proof}

Now we state the main result of this section.    

\begin{theorem}\label{uniqueness of periodic attractor}
For every $(a,b) \in \mathcal H, f_{a,b}$ has exactly one attracting cycle.  
\end{theorem}

\begin{proof}
The result is obvious if $f_{a,b}$ has an attracting fixed point, since $f_{a,b}$ has constant slope on $I_-$. 
So assume that $f_{a,b}$ has two attracting periodic orbits $\mathcal{O}_1,\mathcal{O}_2$ of periods $\geq 2 $, for some  $(a,b) \in \mathcal H$. Let $x \in \mathcal{O}_1$ and $y \in \mathcal{O}_2$ be the be points closest to $1/2$ within $I_-$. By lemma \ref{0.5 in closure}, $1/2$ in the closure of immediate basin components $B_x$ of $x$ and $B_y$ of $y$. Note that immediate basin components of an attracting periodic cycle  are disjoint. Since $x,y \in I_-$, $B_x$ and $B_y$ must intersect, which is impossible. 
\end{proof}

One may also consider the right endpoint of the immediate basin of the rightmost periodic point in $I_-$ and using the same ideas from last proof show that $0$ is in the immediate basin of attraction. Incorporating this with the facts from Lemma \ref{0.5 in closure},  we obtain the following. 
\begin{lemma}\label{0,1 in basin}
    For $(a,b)\in \mathcal{H}$ both $0,1/2$ belong to the immediate basin of attraction of $f_{a,b}$.

\end{lemma}

\subsection{\textbf{Definition of Type}}
 Authors in \cite{MR} first introduce a lift of the map ignoring the $ \sslash1 $. We will follow the same procedure.
 
 Denote the map $\Tilde{f}_{a,b}: I \to \mathbb{R}$ by,
\[
\Tilde{f}_{a,b}(x)=\begin{cases}
              \left(2x+a+\dfrac{b}{2} S(x) \right)+1 &\text{  if $(a-b/2) \in [-1/2,0)$},\\\\
              \left(2x+a+\dfrac{b}{2} S(x) \right)&\text{  if $(a -b/2) \in [0,1)$}.\end{cases}
              \]

Notice $\Tilde{f}_{a,b}$ is just $f_{a,b}$ with the $'\sslash 1'$ removed. Now define $F_{a,b}: \mathbb{R} \to \mathbb{R}$ as,
\[
F_{a,b}(x)=\Tilde{f}_{a,b}(x-k)+2k=\Tilde{f}_{a,b}(x\sslash 1)+2k, \text{ if } x \in [k,k+1], k \in \mathbb{Z}.
\]
The reason for such a definition of $\Tilde{f}_{a,b}$ is because we want that $F_{a,b}(0) \in [0,1)$. Since two lifts differ by an integer, this definition of $\Tilde{f}_{a,b}$ ensures an unique choice of  $F_{a,b}$.\\
\indent To work in general setting the authors  abstract out the following properties of the lift for the Double Standard Family in \cite{MR} which are also true for PLPDM family. Here we state these results from \cite{MR} as lemmas.
\vspace{.3cm}
\begin{lemma} The following are true for $F_{a,b}$ as defined above.
\begin{enumerate}
\item  $F_{a,b}$ is increasing in $x$,
\item  $F_{a,b}(x+k)=F_{a,b}(x)+2k$, for any $k \in \mathbb{Z}$, and $x\in \mathbb R$,
\item $F_{a,b}$ is continuous as function of $(a,b,x)$.
\end{enumerate}
  
\end{lemma}
We will now state a lemma that will help us transition between $f_{a,b}$ and $F_{a,b}$.
\begin{lemma} \label{lift}
The following are true for a circle map and its lift.
    \begin{itemize}
        \item   Let $F: \mathbb{R} \to \mathbb{R}$ be the lift of $f: \mathbb{T} \to \mathbb{T}$. Let $J \subset \mathbb{T}$ then $f(x)\in J$ iff, $F(x)+k \in J$ for some fixed $k \in \mathbb{Z}.$ 
        \item Observe that  $F_{a,b}$ is a lift of $f_{a,b}$, i.e.  $F_{a,b}(x)\sslash1=f_{a,b}(x\sslash1)$ and hence $F^n_{a,b}(x)\sslash1=f^n_{a,b}(x\sslash1)$. 
        \item Observe that $F_{a,b}^n(x)=\Tilde{f}_{a,b}(F_{a,b}^{n-1}(x)\sslash1)+2[F_{a,b}^{n-1}(x)]$, where $y \mapsto [y]$ is the greatest integer function. 
        \item $F_{a,b}$ is a degree 2 lift of $f_{a,b}$, i.e. $F_{a,b}^n(x+k)=F^n_{a,b}(x)+2^nk, n \in \mathbb{N}, k \in \mathbb{Z}, x \in \mathbb{R}.$
    \end{itemize}
  
\end{lemma}

\bigskip
Next lemma establishes that $f_{a,b}$ is semiconjugate to the doubling. The proof can be found  in \cite{MR}.
\begin{lemma}\label{Semiconj}
     Let 
\[
     \Phi_{a,b} (x)=\lim_{n \to \infty}\dfrac{F_{a,b}^n(x)}{2^n}\quad \text{for } x\in \mathbb R.
 \]
\begin{enumerate}
\item The limit exists uniformly in $x$.
\item $\Phi_{a,b} $ is increasing in $x$.
\item $\Phi_{a,b}(x+k)=\Phi_{a,b}(x) +k, k \in \mathbb{Z}$.
\item $\Phi_{a,b}$ semiconjugates $F_{a,b}$ by multiplication by 2, i.e.~$\Phi_{a,b}\circ F_{a,b}(x)=2\Phi_{a,b}(x)$ for $x\in \mathbb R$.
\item $\Phi_{a,b}$ is a lift of a unique $\phi_{a,b}$  that semiconjugates $f_{a,b}$ with the doubling map $D(x)=2x \sslash 1$ i.e.
\[ \phi_{a,b}\circ f=D \circ \phi_{a,b}. 
\]
\end{enumerate}

 \end{lemma}

We can explain the interrelations by a
schematic diagram in Figure \ref{lifts}.

\begin{lemma}\label{cont semiconjugacy}
    For $(a, b)\in \mathcal{P}$, let $\phi_{a,b}$ be the unique semiconjugacy between $f_{a,b}$ and the doubling map. Then $(a,b,x) \mapsto \phi_{a,b}(x)$ is a continuous mapping.

\end{lemma}

\begin{proof}
    Since the convergence of the sequence $\{\Phi_n\}$
in the proof of Lemma \ref{Semiconj} (see \cite{MR})
is uniform independently from $(a, b)$ and since the correspondence $(a, b) \mapsto f_{a,b} $ is
continuous, the result follows.
\end{proof}
\begin{lemma}\label{periodic to periodic}
If $p$ is a periodic point of $f_{a,b}$ of period n then
$\phi_{a,b}(p)$ is a periodic point of $D$ of period $n$. Also $\phi_{a,b}$ is constant on each component of the immediate basin of attraction if $f_{a,b}$ happens to have to one. 
\end{lemma}
\begin{proof}
It follows from semiconjugacy that, $\phi_{a,b}\circ f^n_{a,b}=D^n\circ\phi_{a,b}$. Assume $p$ is a periodic point of $f_{a,b}$ of period $n$, then $f_{a,b}^n(p)=p$ and $\phi_{a,b}(p)=D^n(\phi_{a,b}(p))$. Then follow the proof of Lemma 3.2 of \cite{MR}. \\
\indent Assume $J$ to be some component of the basin of attraction containing the attracting cycle. Say $\phi_{a,b}$ is not constant on $J$ then $\phi_{a,b}(J)$ is non trivial interval.  Since $\phi_{a,b}\circ f^n_{a,b}(J)=D^n\circ\phi_{a,b}(J)$, the right hand side then eventually covers the whole $\mathbb{T}$ for large $n$ and the left hand side accumulates on a cycle, hence contradiction.
\end{proof}

\begin{center}
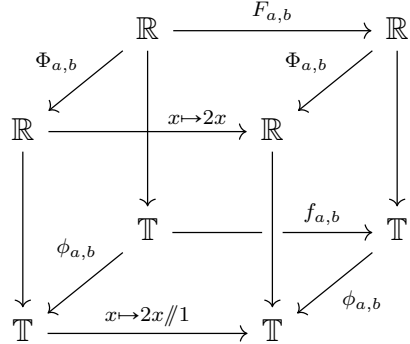

 \begin{tikzcd}
& \mathbb{R} \arrow[dl,swap,"\Phi_{a,b}"] \arrow[rr,"F_{a,b}"] \arrow[dd] & & \mathbb{R}\arrow[dl,swap,"\Phi_{a,b}"] \arrow[dd] \\
\mathbb{R} \arrow[rr,near end, "x \mapsto 2x"] \arrow[dd] & & \mathbb{R} \\
& \mathbb{T} \arrow[dl,swap,near start, "\phi_{a,b}"] \arrow[rr,near end,"f_{a,b}"] & & \mathbb{T} \arrow[dl,"\phi_{a,b}"] \\
\mathbb{T} \arrow[rr,"x \mapsto 2x\sslash1 "] & & \mathbb{T} \arrow[from=uu, crossing over]\\
\end{tikzcd}
\captionof{figure}{\small Schematic diagram of semiconjugacy}\label{lifts}
\end{center}

\vspace{.4cm}

Our goal now is to  define \textit{type} for PLPDM family by importing similar ideas from \cite{MR}.  Let $S$ be a component  corresponding to a single minus itinerary $\sigma=(\sigma_0,\ldots,\sigma_{p-1})$ for the cycle $(x_0,...,x_{p-1})$ with $\sigma_0=-$. Then by Lemma \ref{periodic to periodic}, $\phi_{a,b}(x_0)$ is a periodic point of $D$ of period $p$, hence of the form $k/(2^p-1)$ with $ k \in \{0,\ldots 2^p-2\}$. Let us call $x_0$ as $x_0(a,b)$ to indicate the dependence on $a,b$. Then, 
    \begin{itemize}
        \item $x_0(a,b) \in I_-$,
        \item $f_{a,b}^p(x_0(a,b))=x_0(a,b)$,
        \item $f_{a,b}^i(x_0(a,b)) \neq x_0(a,b)\text{ for }  i=1,\ldots,p-1$.
    \end{itemize}
   
We will first show the map, \[
\Xi: S \to \mathbb{T} ; \Xi(a,b)= x_0(a,b)
\]
is continuous. 
\begin{claim}\label{imlicit}
    $\Xi$ is continuous.
\end{claim}
    \begin{proof}
        We use Theorem 4.1 of \cite{BGD} (implicit function theorem for continuous functions) to show this. Since $\mathbb{T} \times I$ is path connected and and $\mathbb{T}$ is oriented manifold. Now define 
        \[
        \Theta : (\mathbb{T} \times I) \times \mathbb{T} \to \mathbb{T}; \text{ by, }
        \Theta(a,b,x)=f_{a,b}^p(x)-x.
        \]
        The map $\Theta$ is continuous and $\Theta(a,b,x_0(a,b))=0$ if  $(a,b) \in S $.
        Now for all $(a,b)$,  $x \mapsto \Theta(a,b,x)$ is open, discrete (see \cite{BGD} for definition). Now, as $f_{a,b}^p$ has slope $<1$ on some neighbourhood $V(x_0(a,b)) \subset \mathbb{T}$ around $x_0(a,b)$ so we have that $\Theta|_{V(x_0(a,b))}: V(x_0(a,b)) \to \mathbb{T}$ is a linear homeomorphism onto its image. This means that deg$((\Theta_{a,b}), x_{a,b})=1$ (although it should be $(\Theta_{(a,b)}, x_{a,b})$ by following the notation in \cite{BGD}. The reader can refer to the remark in the paragraph above the Proposition 2.30 in \cite{AlH}). Then by Theorem 4.1 of \cite{BGD} there exists a neighbourhood $U_{a,b} \subset \mathbb{T} \times I$ for which there is some continuous function $X: U_{a,b} \to \mathbb{T}$
        \vspace{.1cm}
        
        \noindent so that $\Theta(\Tilde{a},\Tilde{b},X(\Tilde{a},\Tilde{b}))=\Theta(a,b,X(a,b))=0$, or, $f^p_{\Tilde{a},\Tilde{b}}(X(\Tilde{a},\Tilde{b}))=X(\Tilde{a},\Tilde{b})$ for all $(\Tilde{a},\Tilde{b}) \in U_{a,b}$. The other $p-1$ conditions are satisfied by choosing $U_{a,b}$ small enough by using continuity. We are now done with our claim.
\end{proof}

\noindent The following theorem makes it possible to define a \textit{type} for tongues.
\begin{theorem}\label{Tongue open}
    If $S$ is a tongue of  order $p$ with single minus itinerary  $\sigma=(\sigma_0,\ldots,\sigma_{p-1})$ with $\sigma_0=-$. Then the following  map is constant.
    \[
    \tau : S \to \mathbb{T}, \tau(a,b)=\phi_{a,b}(x_0)
    \]
    where $(x_0,x_1,...,x_{p-1})\in \mathbb{T}^p$ is the unique periodic orbit of exact period $p$ of $f_{a,b}$ and has a single minus itinerary with $x_0 \in I_-$.
\end{theorem}
\begin{proof}
  
The well definedness of the map is clear from Lemma \ref{periodic to periodic}. Notice that  $S$ is connected and codomain is the discrete set, $\{k/(2^p-1): 0 \leq k \leq 2^p-2 \}$. We can now conclude that $(a,b) \mapsto \phi_{a,b}(x_0(a,b))$ is constant using Lemma \ref{cont semiconjugacy}.
\end{proof}

\begin{definition}[Type]
By Theorem \ref{uniqueness of periodic attractor},     for $(a,b) \in \mathbb{T}\times [0,1]$, $f_{a,b}$ has  atmost one attracting periodic orbit $ (x_0,x_1,...,x_{p-1})\in \mathbb{T}^p$ of some exact period
    $p \in \mathbb{N}$ with a single minus itinerary with $x_0 \in I_-$. Let $\phi_{a,b}$ be the semiconjugacy with the doubling map and let $\tau=\phi_{a,b}(x_0)$. Then the orbit $(x_0,x_1,...,x_{p-1})$ is said to have type $\tau$. In this case we will also say that  the parameter $(a, b)$ have type $\tau$.  \\
   \indent 
   By Theorem \ref{Tongue open}, type $\tau$ remains constant for parameters inside a given tongue. Define this constant type $\tau$ as the type of the tongue.  

\end{definition}

\noindent Observe that using ideas from Claim \ref{imlicit}, the following is immediate.
\begin{proposition}\label{stable componenet open}
    $\mathcal{H}$ is open in $\mathbb{T}\times I.$
\end{proposition}

\section{Existence of Tongues}

In this section will prove the existence of parameters with single minus attracting cycle of order 3 and 4. First we do the simpler period 3 case. We  start by proving a lemma.
 \begin{lemma}\label{no only minus}
     Let $\sigma=(-,\sigma_1,\sigma_2,...,\sigma_p) $ be the itinerary of an attracting periodic orbit of period p. Then either p = 1 or there is at least one $+$ in $\sigma=(-,\sigma_1,\sigma_2,...,\sigma_p) $
 \end{lemma}

\begin{proof}
The proof is by dividing into three cases.
\begin{itemize}
    \item[Case 1:] Let $f_{a,b}(1/2)\in I_+$, then orbit of any point $x \in I_-$ monotonically decreases  to enter $I_+$ after finitely many iterations (otherwise $f_{a,b}$ will have a fixed point in $I_-$, which is not possible in this case).
    
    \item[Case 2:] Assume $f_{a,b}(1/2)\in I_-$ and $f_{a,b}$ intersects the diagonal. In this case whole $I_-$ is the immediate basin of attraction of this attracting fixed point. So given $x \in I_-$ its orbit is an infinite sequence of distinct points  converging to the attracting fixed point. 
    
    \item[Case 3:] Assume $f_{a,b}(1/2)\in I_-$ and $f_{a,b}$ does not intersect the diagonal. In this case, for $x \in I_-$ either $f_{a,b}(x) \in I_+$ or the orbit of $x$ monotonically increases to jump backwards in $I_+$ after finitely many steps (otherwise, $f_{a,b}$ has a fixed point in $I_-$, which is not possible).     
\end{itemize}
The cases are illustrated in the figure \ref{cases}.

\end{proof}

\begin{figure}[hbt!]
    \centering
    
    \begin{subfigure}[t]{0.4\textwidth}
        \centering
        \includegraphics[width=\textwidth]{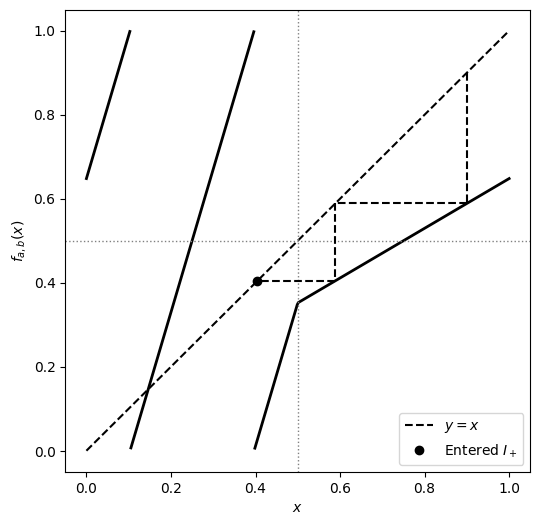}
        \caption{Case 1}
    \end{subfigure}
    \hfill
    \begin{subfigure}[t]{0.4\textwidth}
        \centering
        \includegraphics[width=\textwidth]{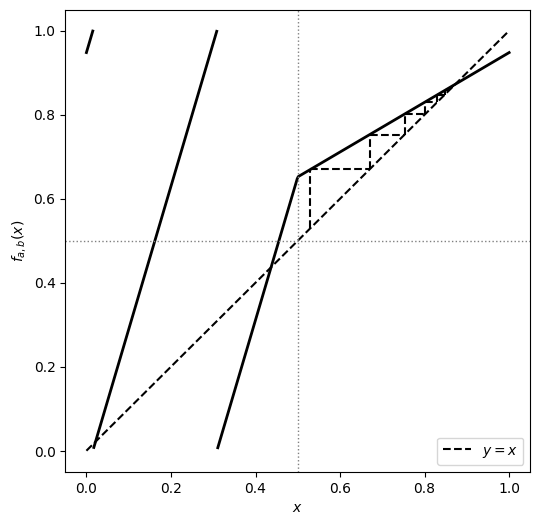}
        \caption{Case 2}
    \end{subfigure}
    \hfill
    \begin{subfigure}[t]{0.4\textwidth}
        \centering
        \includegraphics[width=\textwidth]{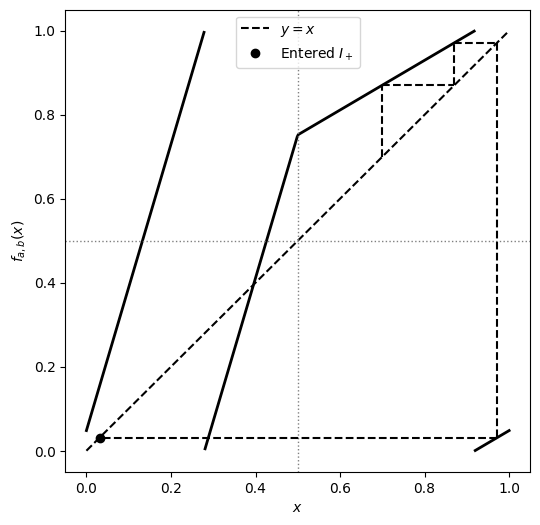}
        \caption{Case 3}
    \end{subfigure}
    
    \caption{Cobweb diagrams for three different cases.}
    \label{cases}
\end{figure}

\begin{corollary}
     Let $f$ in PLPDM family such that $f(I_-)\subset I_-$, then $f$ has  no attracting periodic point of period greater than or equal to 2.
\end{corollary}

\subsection{\textbf{Period 3 tongue}}

\vspace{.3cm}

We will now show that a tongue of period 3 exists. The method used here specific but of a similar idea with period 4 case. It is just that  3 is a period too small to illustrate the general scheme and for that we will prove the period 4 case after this.

\begin{lemma}\label{int line}
     $y=cx+T$ restricted to $J$, intersects diagonal $y=x$ if $T \in (1-c)J$.
\end{lemma}
\begin{lemma}\label{good interval period 3}
     Fix $\epsilon $ small enough so that  $B^-B^{+{2}}<1$ and let $b=1-\epsilon$.  Then there is an interval $I_b$  so that $f_{a,b}(1/2)>\dfrac{2-A^+}{B^+}$ if $a \in I_b$.  
     
\end{lemma}
\begin{proof}
    Observe that $f_{a,b}(\frac{1}{2})=\dfrac{B^+ +2A^+ -4}{2}$. Plugging in the values in the required inequality and simplifying we choose,
\begin{align*}
    0<\delta < \dfrac{1-10\epsilon+4\epsilon^2}{10-4\epsilon}
\end{align*} 
where $A^+=1/2-\delta,b=1-\epsilon$. Hence this translates to,
\begin{align*}
    a \in  \left(\dfrac{2-\epsilon}{2}-\dfrac{1-10\epsilon+4\epsilon^2}{10-4\epsilon},\dfrac{2-\epsilon}{2}\right)=I_b. 
\end{align*}
 
\end{proof}
The point of of having  $f_{a,b}(1/2)>\dfrac{2-A^+}{B^+}$, is the following:  Since $A^+ \in (0,1/2)$ we will get that $f_{a,b}(x)\in \left(\dfrac{2-A^+}{B^+},1/2\right)$ and $f_{a,b}^2(x) \in I_+$(as $b \lesssim 1$). Here $\left(\dfrac{2-A^+}{B^+},1/2\right)$ is the last lap in $I_+$  and $f_{a,b}$ maps $\left(\dfrac{2-A^+}{B^+},1/2\right)$ into $I_+$ as a linear diffeomorphism with slope $B^+$ (see figure \ref{3 case}). This condition ensures that any $x \in I_-$ has itinerary $(-,+,+,...)$. 
 
Then for $a \in \left(\dfrac{2-\epsilon}{2}-\dfrac{1-10\epsilon+4\epsilon^2}{10-4\epsilon},\dfrac{2-\epsilon}{2}\right)=I_b$, $f_{a,b}^3$ has an attracting cycle with itinerary $(-,+,+)$.\newline
\begin{proposition}\label{existence of period 3 tongue}

     Let $\epsilon,I_b$ as in Lemma \ref{good interval period 3}. Then $ \forall a\in I_b$, $f_{a,b}$ has an attracting periodic point of period 3 with single minus itinerary.
\end{proposition}
\begin{proof}
    Let $k_\epsilon=\dfrac{1-10\epsilon+4\epsilon^2}{10-4\epsilon} \text{ and we already have that } I_b= \left(\dfrac{2-\epsilon}{2}-\dfrac{1-10\epsilon+4\epsilon^2}{10-4\epsilon},\dfrac{2-\epsilon}{2}\right)$. 
Then $A^+ \in (\frac{1}{2}-k_\epsilon,\frac{1}{2})$ for $a \in I_b$. This means $f_{a,b}(\frac{1}{2})>\dfrac{2-A^+}{B^+}$ (by Lemma \ref{good interval period 3}). Hence $x\in I_-$ has itinerary $(-,+,+,...)$ and 
\begin{align*}
    f_{a,b}^3(x)=B^-B^{+^{2}}x+\{P_1(b)a+P_2(b)\} \mod 1.
\end{align*}
 So formally there is only one$\mod 1$, as $f_{a,b}^2(I_-) \subset I_+$.
    By simple calculation , $P_1(b)=B^{+{2}}+B^++1\lesssim 21 $. Say $P_1(b)>20$ for $\epsilon$ small enough. 
    
let $c=B^-B^{+^{2}}$, $ J=\left(\dfrac{2-\epsilon}{2}-k_\epsilon,\dfrac{2-\epsilon}{2}\right),$  $T=P_1(b)a+P_2(b)$. Hence $T$ covers an interval of length $P_1(b). k_\epsilon \sim 20\times 1/10=2 >1$. So there will be a subinterval $I_1\subset I_b$, so that $P_1(b).|I_1| =(1-c)|J|$. And hence we have that $f_{a,b}^3(x)=x$ by Lemma \ref{int line}.
\end{proof}

In figure \ref{3 case} we give an example of  parameters $a,b$ in the range suggested by the proof of Proposition \ref{existence of period 3 tongue}. 
\begin{figure}[hbt!]
    \centering
    \includegraphics[scale=.6]{ 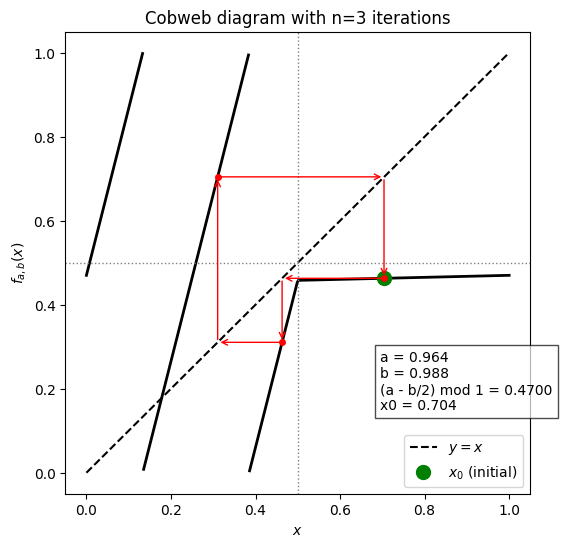}
    \caption{\small $a=0.964, b=0.988, A^+=.47$}
    \label{3 case}
    
\end{figure}

\subsection{\textbf{Tongues of period $\geq 4$}}
\vspace{.4cm}

We will now give a series of lemmas that will eventually help us in proving that tongue exists for period 4. 
\subsubsection{\textbf{Idea of the proof for period 4}}
\vspace{.2cm}

There are two things one needs to ensure when looking for a single minus, period 4, attracting orbit of some $x$.
\begin{itemize}
    \item 
    Firstly,  $x \in I_-, \text{ }f_{a,b}(x)\in I_+, \text{ } f_{a,b}^2(x)\in I_+, \text{ } f_{a,b}^3(x)\in I_+ $,
    
    \item and $f_{a,b}^4(x))=x$.
\end{itemize}

\indent That $x$ is of minimal period 4 is taken care of by the itinerary itself. We will work with the lift $F_{a,b}$ and use the continuity of $(a,b,x) \mapsto F_{a,b}(x)$ and the fact that $a \mapsto F_{a,b}(x)$ is a translation. We will first look for  $a',x$ such that  $x$ has the desired itinerary and there is a big interval $J$ about $a'$ such for all $a\in J$ the same chosen $x$ has the desired itinerary. This is summarized in Lemmas \ref{setting up1}, \ref{set up 2}, \ref{nest}.

\vspace{.2cm}

\indent Note that $F_{a,b}^n$ is a piecewise  linear map in $a$ variable. If one can inductively maintain the required itinerary then the expansions (slopes) in $a$ variable are respectively $1,   (B^++1), (B^{+^{2}}+B^++1), (B^{+^{3}}+B^{+^{2}}+B^++1)$ for the first four iterations. Since $b\lesssim1$, this rate is large enough to ensure the existence of a big interval $J$ with the correct itinerary. The graph of $F^n_{a,b}$ goes up or down accordingly as $a$ increases or decreases. Now if there is good  expansion in $a$ at the 4th stage  and enough room to vary $a$ in $J$ then the graph can be vertically translated to intersect the translated diagonal at the point $(x,x+k)$ for some integer $k$ and the task is completed. The required estimates on the size of the interval $J$ is explained by Lemma \ref{use int line}.   \\

\subsubsection{\textbf{Proof of existence of tongues for period 4}}
\vspace{.3cm}
 
  We will work with the unique lift of $f_{a,b}$ so that $F_{a,b}(0) \in [0,1]$. So assume $A^+=a-b/2 \in (0,\tfrac{1}{2})$, i.e. $b/2<a<1/2+b/2$. Before we start our next proposition we will need to set up some notation for brevity,
 \begin{itemize}
    
    \item Define $\xi_i(=\xi_i(b)), i=1,2,3$ as follows,  \[
 \xi_1=\dfrac{1}{B^{+^{3}}+B^{+^{2}}+B^++1}, \text{ } \xi_2= (B^++1) \xi_1, \text{ }\xi_3=(B^{+^{2}}+B^++1)\xi_1.
 \]
    \item Define the following maps,
\[
\theta_{i}^{b,x}: \mathbb{T} \to \mathbb{R}, \text{ }\theta_{i}^{b,x}(a)=F_{a,b}^{i}(x); \text{ } i=1,2,3,4.
\]
    \item For $b\in [1/2,1]$, define, \[ \text{SM}^{4}(b) \coloneq \{ (a,x) \in \mathbb{T}^2 :x \in I_-, \text{ }f_{a,b}(x)\in I_+, \text{ } f_{a,b}^2(x)\in I_+, \text{ } f_{a,b}^3(x)\in I_+ \} \](SM stands for \textit{starts with minus}).
 \end{itemize}
 \vspace{.1cm}
 For $z,w \in \mathbb{R}/\mathbb{Z}$, $z \wedge w \coloneq \min\{z,w\}$ where $\mathbb{R}/\mathbb{Z}$ is identified with $[0,1)$ equipped with the natural order from $\mathbb{R}$. For $x,y \in \mathbb{R}$, $x \wedge y \coloneq \min\{x,y\}$.
\bigskip

\begin{lemma}\label{setting up1}
    Let $b\lesssim1$ be such that there is some $(a,x)\in \text{SM}^4(b)$ with $\dfrac{2-b}{2}<a <1/2+b/2$ and $F_{a,b}(I_-)\subset (2,2.5)$.
    Pick such an $(a,x)\in \text{SM}^{4}(b)$ and fix it. 
    \begin{enumerate}
        \item Let $\alpha \in \mathbb{T}$ so that $|\alpha-a| \leq \text{ } f_{a,b}(x) \wedge (1/2-f_{a,b}(x))$ then $f_{\alpha,b}(x)\in I_+$.
        \vspace{.3cm}
        
        \item Let $\alpha \in \mathbb{T}$ so that $\alpha$ satisfies the bound in (1) (i.e. $F_{a,b}(x) \in (2,2.5) $ ) 
        and $|\alpha-a| \leq \text{ } \dfrac{f^2_{a,b}(x) \wedge (1/2-f^2_{a,b}(x))}{B^++1} $ then $f^2_{\alpha,b}(x)\in I_+$.
          \vspace{.3cm}
        \item Let $\alpha \in \mathbb{T}$ so that $\alpha$ satisfies bounds in  (1), (2) and  $|\alpha-a| \leq \text{ } \dfrac{f^3_{a,b}(x) \wedge (1/2-f^3_{a,b}(x))}{B^{+^{2}}+B^++1} $ then $f^3_{\alpha,b}(x)\in I_+$.
        
    \end{enumerate}
\end{lemma}
\begin{proof}
Observe that $0<F_{a,b}(x)<2.5$ for $x\in I$ and $0<F_{a,b}(A^+)< 2.5$ under the assumed conditions. For $\alpha \in \mathbb{T}$, denote $\alpha^+=\alpha-b/2, \alpha^-=\alpha+3b/2$, then the following hold. Since $\dfrac{2-b}{2}\geq b/2$ if $b \in [0,1] $, we get that $b/2<a<1/2+b/2$. 
 
 \vspace{.3cm}
 
\begin{enumerate}
    \item Under the assumptions on $b$, $f_{a,b}(I_-) \subset I_+$ and  $F_{a,b}(x) \in (2,2.5)$ . Then  $F_{a,b}(x)-2=f_{a,b}(x)$ and $2.5-F_{a,b}(x)=0.5-f_{a,b}(x)$.
    Since $F_{a,b}(x)-F_{\alpha,b}(x)=a-\alpha$ and (1) follows.
    
    \vspace{.3cm}
    \item Note that $f^2_{a,b}(x) \in I_+$ if,  $F^2_{a,b}(x)  \in (4,4.5 ) \bigcup (5,5.5)\bigcup (6,6.5)$, $|F_{a,b}^2(2,2.5)|= 2-(1-b)$.
    \vspace{.1cm}
\begin{itemize}
    \item Now $F^2_{a,b}(x)  \in (4,4.5 )$ and if  $|\alpha-a|\leq \dfrac{(4.5 -F^2_{a,b}(x) ) \wedge (F^2_{a,b}(x)-4)  }{(B^++1)}$ then,\\ $F^2_{\alpha,b}(x) \in (4,4.5) $. Observe that,
    \vspace{.3cm}
    \begin{itemize}
        \item[]  $F^2_{a,b}(x)  \in (4,4.5 ) \implies 4.5 -F^2_{a,b}(x) = 1/2-f_{a,b}^2(x) \text{ as } F^2_{a,b}(x)-4 = f^2_{a,b}(x) $.
        \vspace{.3cm}
        \item[] So, $\dfrac{(4.5 -F^2_{a,b}(x))  \wedge( F^2_{a,b}(x)-4)  }{(B^++1)}=\dfrac{(1/2 -f^2_{a,b}(x))  \wedge f^2_{a,b}(x) }{(B^++1)}.$ 
    \end{itemize}
   \vspace{.3cm}
    \item  We can similarly deal with $F^2_{a,b}(x)  \in (5,5.5 ) $ and $F^2_{a,b}(x)  \in (6,6.5 ) $
    
\end{itemize}
    Now say $F_{a,b}^2(x) \in (4,4.5)$ and let $\alpha \in \mathbb{T} $ such that it satisfies (1) and
    \vspace{.3cm}
    
    $|\alpha-a| \leq \text{ } \dfrac{f^2_{a,b}(x) \wedge (1/2-f^2_{a,b}(x))}{B^++1}$. Because $\alpha$ satisfies (1), we have $F_{\alpha,b}(x) \in (2,2.5)$. Hence, 
    $F^2_{a,b}(x)=F_{a,b}(B^-x+A^-)=B^+(B^-x+A^--2) +A^++4=$ and $F^2_{\alpha,b}(x)=F_{\alpha,b}(B^-x+\alpha^-)=B^+(B^-x+\alpha^--2) +\alpha^++4$, or, $F^2_{\alpha,b}-F^2_{a,b}= (B^++1)(\alpha-a)$. We can Deal with the other cases similarly.
    
    \vspace{.3cm}

    \item Now $f_{a,b}^3(x) \in I_+$ means, \[
    F_{a,b}^3(x) \in \bigcup_{j=8}^{14}(j,j+1/2).
    \]
    Without loss of generality consider, $F_{a,b}^3(x) \in (8,8.5)$. Assume $\alpha \in \mathbb{T}$ satisfies (1),(2) then as before, $F_{a,b}^3(x)-F_{\alpha,b}^3(x)=(B^{+{2}}+B^++1)(a-\alpha)$. Similar considerations for other cases makes the proof (3) complete.
\end{enumerate}
 
\end{proof}
  Let $b\lesssim1$ be such that $\text{SM}^4(b) \neq \emptyset$
    Pick some $(a,x)\in \text{SM}^{4}(b)$ and fix it. Define \[
    \gamma_1(a,x)=f_{a,b}(x) \wedge (1/2-f_{a,b}(x)),\]
    \[\gamma_2(a,x)=\dfrac{f^2_{a,b}(x) \wedge (1/2-f^2_{a,b}(x))}{B^++1},\]
   \[ \gamma_3(a,x)= \dfrac{f^3_{a,b}(x) \wedge (1/2-f^3_{a,b}(x))}{B^{+^{2}}+B^++1}.
    \]
    Then define \[\gamma(a,x)=\min \{\gamma_1(a,x),\gamma_2(a,x),\gamma_3(a,x),(x-0.5),(1-x)\}.\] This definition will be useful for final proof.\\
\begin{lemma}\label{set up 2}
        There exists $\epsilon>0$ such that for all $b \in (1-\epsilon,1)$ there exists some $a,x \in \mathbb{T}$ such that following are true. 
        \begin{enumerate}
            \item $x \in (1/2+\xi_1, 1-\xi_1).$
            \vspace{.2cm}
            \item $\theta_1^{b,x}(a) \in (2+\xi_1,2.5-\xi_1)=K^1_{1}.$
            \vspace{.2cm}
            \item $\theta_2^{b,x}(a) \in \bigcup_{j=1}^{3}(3+j+\xi_2,3+j+1/2-\xi_2) $. Call, $K^2_{j}=(3+j+\xi_2,3+j+1/2-\xi_2); j=1,2,3.$
            \vspace{.2cm}
            \item $\theta_3^{b,x}(a) \in \bigcup_{j=1}^{j=7}(7+j+\xi_3,7+j+1/2-\xi_3)$. Call $K^3_{j}=(7+j+\xi_3,7+j+1/2-\xi_3), j=1,...,7.$
        \end{enumerate}
        In particular for such an $(a,x)\in \mathbb{T}^2$ we have that $(a,x)\in \text{SM}^4(b)$ with $\gamma(a,x)>\xi_1$.
    
    \end{lemma}
\begin{proof}
     Let $x \in (1/2+\xi_1,1-\xi_1)$ then by letting $b \to 1$ it is clear that
        \[
        (1/2+\xi_1-B^-x,1-\xi_1-B^-x) \subset \displaystyle\Bigl((2-b)/2,b/2+1/2\Bigr).
        \]

Also,
       \[
       \theta_1^{b,x}:  (1/2+\xi_1-B^-x,1-\xi_1-B^-x) \to  (1/2+3b/2+\xi_1,1+3b/2-\xi_1)
       \]
       
 \noindent    is an onto linear translation by $(3b/2+B^-x)$. If $a \in  (1/2+\xi_1-B^-x,1-\xi_1-B^-x) $ then  $b/2<a<1/2+b/2$, which ensures that the correct lift $F_{a,b}$ is being applied in accordance with our calculations.
 
Observe that $(a,b,x) \mapsto F_{a,b}(x)$ is continuous. Now the result follows trivially letting $b \to 1$ and using Lemma \ref{nest}.
\end{proof}

In what follows we will assume $b=1$. Then,

\[
\xi_1 = \tfrac{1}{85}, \quad \xi_2 = \tfrac{1}{17}, \quad \xi_3 = \tfrac{21}{85},
\]

\[
|K^1_{1}|=|(0.5+\xi_1,1-\xi_1)|=\tfrac{81}{170}, 
\quad |K^2_{j}|=\tfrac{13}{34}, 
\quad |K^3_{j}|=\tfrac{1}{170}.
\]
\vspace{.01cm}

Let $ d(A,B) \text{  is  the distance between two sets in metric space } (X,d)$. Then,  
 
\[
d(K^2_{j},K^2_{j+1})=\tfrac{21}{34}, 
\quad d(K^3_{j},K^3_{j+1})=\tfrac{169}{170}.
\]

\[
\sup\{|x-y|: x \in K^2_{j}, y \in K^2_{j+1}\} = \tfrac{47}{34}, \quad j=1,2.
\]

\[
\sup\{|x-y|: x \in K^3_{j}, y \in K^3_{j+1}\} = \tfrac{171}{170}, \quad j=1,\ldots,6.
\]

\begin{lemma}\label{nest}
Let $b=1$. The following points outline a nesting procedure. Let $x \in (1/2+\xi_1,1-\xi_1)$ be arbitrary.
\begin{enumerate}
    \item There exists nonempty interval $J_1 \subset (1/2,1) \subset \mathbb{T}$ such that, $\theta_1^{b,x}|_{J_1}: J_1 \to \mathbb{R}$ is a linear translation onto $K^1_{1}$.  Hence, $|J_1|=|K^1_{1}| = \tfrac{81}{170}$.
     To be precise, $J_1=(1/2+\xi_1, 1-\xi_1)$ is mapped onto $(2+\xi_1,2.5-\xi_1)$ by $\theta_1^{b,x}$ as a translation by $1.5$.
    \item  There exists non empty interval $J_2 \subset J_1$ so that, $\theta_2^{b,x}(J_2)$ =$K_{j}^2$ for some $j \in \{1,2,3\}$ and $|J_2| \sim \tfrac{1}{5} \times \tfrac{13}{34}.$
    \item There exists a non empty interval $J_3 \subset J_2$ so that $\theta_3^{b,x}(J_3)$ =$K_{j}^3$ for some $j \in \{1,...,7\}$ and $|J_3| \sim \tfrac{1}{21} \times \tfrac{1}{170}.$
\end{enumerate} 

\end{lemma}
\begin{proof}
    The proof is straightforward calculation if one remembers that under the conditions, $\theta_2^{b,x}|_{J_1}$ and $\theta_3^{b,x}|_{J_1} $ are  linear maps with slope $5$ and slope $ 21$ respectively and uses the information just above the lemma's statement.
\end{proof}

\begin{lemma}\label{use int line}
     Let $b\lesssim 1$ be such that $SM^4(b) \neq \emptyset$.
    If there exists some $(a,x) \in SM^4(b)$ such that $\gamma(a,x)> \xi_1$. Then $f_{a,b}$  has a order 4 periodic point with single minus itinerary.
    
\end{lemma}
\begin{proof}
 Using Lemma \ref{setting up1}, there is an interval $(a-\delta,a+\delta)\subset\mathbb{T}$ with $\delta>\xi_1$ and  $\theta_4^{b,x}|_{(a-\delta,a+\delta)}$ is a line of slope equal to $1/\xi_1$. Since $|(a-\delta,a+\delta)|=2\delta$ and $2\delta \times 1/\xi_1 \gtrsim 2 $, we are done by Lemma \ref{int line}. 
\end{proof}

We are ready for our final result.

\begin{proposition}
    There exists an interval $(1-\epsilon,1)$ so that for all $b \in (1-\epsilon,1) $ there exists an interval $I_b \subset \mathbb{T}$ such that for all $a \in I_b$ there is some  $ x\in I_{-}$ so that $x$ is a single minus itinerary, attracting periodic point of order 4 for $f_{a,b}$.
\end{proposition}
\begin{proof}
    By Lemma \ref{set up 2} there is some $\epsilon$ so that if $b \in (1-\epsilon, 1)$ then hypotheses of Lemma \ref{use int line} is satisfied and we are done.
\end{proof}

\begin{corollary}
    There is a tongue of order 4.
\end{corollary}
\begin{proof}
   By  the very design of the Lemmas \ref{setting up1}, \ref{set up 2}, \ref{nest}, \ref{use int line} the variation of $b$ in the suitable range will result the intervals $I_b$ to vary continuously and their length only increase as $b \to 1$. The tongue is a open, connected component of $\mathcal{H}$ so these $I_b$'s will continuously move while staying inside the component, giving us the result.
   \end{proof}

\subsubsection{\textbf{General algorithm for any $n$ and proof of Theorem \ref{tongue esist}}}
\vspace{.3cm}
Let $n \geq 4$ be a fixed integer and $b \lesssim 1$.
 \begin{itemize}
    
    \item Define $\xi_i(=\xi_i(b)), i=1,2,3,...,n$ as follows,  
\[ \xi_1=\dfrac{1}{B^{+^{n-1}}\ldots+B^{+^{2}}+B^++1}, \text{ } \xi_2= (B^++1) \xi_1, \text{ }\xi_3=(B^{+^{2}}+B^++1)\xi_1,\]
 \[\ldots,  \xi_{n-1}=(B^{+^{n-2}}\ldots+B^{+^{2}}+B^++1)\xi_1.
 \]
    \item Define the following maps,
\[
\theta_{i}^{b,x}: \mathbb{T} \to \mathbb{R}, \text{ }\theta_{i}^{b,x}(a)=F_{a,b}^{i}(x); \text{ } i=1,2,\ldots, n.
\]
\item For $(n-1)\geq m \geq 1$ define 
\[
K_j^{m}=\displaystyle\Bigl((2^m-1+j)+\xi_m \text{ }, \text{ }(2^m-1+j)+1/2-\xi_m\Bigr); \text{ }1 \leq j \leq 2^m-1.
\]

\end{itemize}
\begin{sublemma}\label{1st stage range for a}
If $x \in (1/2+\xi_1,1-\xi_1)$ then 
\[
(1/2+\xi_1-B^-x,1-\xi_1-B^-x) \subset \displaystyle\Bigl((2-b)/2,b/2+1/2\Bigr) \text{  as } b \to 1.
\]
 
\end{sublemma}
Fix some  $b\in (1/2,1)$ such that $B^{+^{n-1}}B^{-}<1$. 
Now we give the algorithm. Let $x \in (1/2+\xi_1,1-\xi_1)$ be arbitrary.  
\begin{itemize}
    \item Take the interval  $J_1=(1/2+\xi_1,1-\xi_1) \subset \mathbb{T}$ ; $|J_1|=1/2-2\xi_1$ so that $\theta_1^{b,x}(J_1)=(2+\xi_1,2.5-\xi_1)$. Observe that $J_1 \subset (b/2,1/2+b/2)$ by Sublemma \ref{1st stage range for a}.
    \item Check if $(B^++1)|J_1|>(1.5-2\xi_2)$. If yes
    then there exists $J_2\subset J_1, |J_2|=\dfrac{1/2-2\xi_2}{B^++1}$ and $\theta_2^{b,x}(J_2)=K^2_j$ , for some $1 \leq j \leq 2^2-1$. 
    \item Check if $(B^{+^{2}}+B^++1)|J_2|>(1.5-2\xi_3)$. If yes then there exists $J_3\subset J_2$, $|J_3|=\dfrac{1/2-2\xi_3}{(B^{+^{2}}+B^++1)}$ and $\theta_3^{b,x}(J_3)=K_j^{3}$ for some $1 \leq j \leq 2^3-1$. 
    \item Check if $(B^{+^{3}}+B^{+^{2}}+B^++1)|J_3|>(1.5-2\xi_4)$. If yes then there exists
    \vspace{.2cm}
    $J_4\subset J_3$, $|J_4|=\dfrac{1/2-2\xi_4}{(B^{+^{3}}+B^{+^{2}}+B^++1)}$ and $\theta_4^{b,x}(J_4)=K_j^{4}$ for some $1 \leq j \leq 2^4-1$. 
    \end{itemize}
    \small{
    \begin{itemize}
        \item [\textbf{.}]
    \end{itemize}
     \begin{itemize}
        \item [\textbf{.}]
    \end{itemize}
     \begin{itemize}
        \item [\textbf{.}]
    \end{itemize}
   
    \begin{itemize}
    \item Check if  $(B^{+^{n-2}}+\ldots+B^{+^{2}}+B^++1)|J_{n-2}|>(1.5-2\xi_{n-1})$. If yes then there exists $J_{n-1}\subset J_{n-2}$ so that  $|J_{n-1}|=\dfrac{1/2-2\xi_{n-1}}{(B^{+^{n-2}}+\ldots+B^{+^{2}}+B^++1)}$ and $\theta_{n-1}^{b,x}(J_{n-1})=K_j^{n-1}$ for some $1 \leq j \leq 2^{n-1}-1$  [$(n-1)$-st stage].

\end{itemize}
\vspace{.5cm}

Now it is clear that given any $n$, we can follow this algorithm if $b$ is close enough to $1$ by observing that,
\[
\dfrac{B^{+^{(j+1)}}+...+B^{+^{2}}+B^++1}{B^{+^{j}}+...+B^{+^{2}}+B^++1}>B^+ \text{, for all } b \in [1/2,1], j \in \mathbb{N} \cup \{0\}.
\]

 Now by formulating lemmas similar to the period 4 case, we have a proof of  Theorem \ref{tongue esist}.

\subsubsection{\textbf{All types are realised}}
\vspace{.4cm}

\begin{lemma}(\cite{D})\label{for all types} The following are true.
    \begin{itemize}
        \item $(a,b,x) \in \mathbb{R}\times [0,1]\times\mathbb{R}$ and $n\in \mathbb{N}$ then
    \[
    F_{a+1,b}^n(x)=F_{a,b}^n(x)+2^n-1.
    \]
    \item Let $n \in\mathbb{N}; a,b \in \mathbb{T}, k <2^n-1$ and $x=(X \text{ mod }1) \in \mathbb{T}$ with  $F^n_{a,b}(X)=X+k$, then $\phi_{a,b}(x)=\dfrac{k}{2^n-1}$.
    \end{itemize}
   
\end{lemma}

\begin{theorem}
    Let $\tau$ be a periodic point of $x \to 2x \sslash 1$. Then there exists a tongue with elements consisting of parameters of type $\tau$.
\end{theorem}
\begin{proof}
Observe that $a \mapsto F^{n}_{a,b}(x)$ is increasing map. The proof follows now from Lemma \ref{for all types}.
\end{proof}

\section{Eyes}\label{eye}

We will now show that a non-single minus itinerary of the distinguished attracting periodic point results in an eye.

\begin{proof}[\textbf{Proof of Theorem \ref{eye combinatorics}}]
    By Theorem \ref{imlicit}, $H$ is open, connected. We will now use ideas from the proof of Theorem \ref{uniqueness of periodic attractor}. The anticlockwise direction is the positive direction. Because $\sigma$ has more than one minus, there is a left endpoint $z(=z(a,b))$  of the basin of the right most attracting periodic point $x_r\in I_-$ about which we can say the following.
    \begin{itemize}
        \item The endpoint $z$ is fixed point of $f_{a,b}^n$ for $(a,b) \in H$.
        \item The point $z$ is inside $I_-$ because there is already one immediate basin containing $1/2$.
        \item Since the boundary of basin components go to boundary under action of $f_{a,b}$, $z$ is  not preperiodic to 1/2. It is  because 1/2 is in the interior of the basin component by Theorem \ref{uniqueness of periodic attractor}. 
    \end{itemize}

      So $z$ is in fact a fixed point of $f_{a,b}^n$ which is repelling from the right.

    Now  for  $(a,b) \in H$  the slope of $f_{a,b}^n$ on $(z,x_r)$ is less than $B^-B^{+^{(n-1)}}$. 
     
    Since $B^-B^{+^{(n-1)}} \to 0$ as $b \to 1$, the point $z$ stops being fixed point of  $f_{a,b}^n$ that is repelling from the right, for $(a,b) \in H$ and $b$ close to 1. So the distance between $H$ and the ceiling of the parameter space $\mathbb{T}\times\{1\} \subset \mathcal P$ is strictly positive, which means $H$ is an eye. 
\end{proof}
\section{Case of  
  $(a,b) \notin \overline{\mathcal{H}}$}
We can also ask the question regarding the existence of the  wandering intervals in PLPDM family. Indeed, wandering intervals don't exist for any map in PLPDM family since they fall in the class $\mathcal{A}$, as defined in \cite{MMS}. 
Then one just needs to apply Theorem A$'$ from \cite{MMS} and conclude.
\begin{proposition}
    Assume $f \in$ PLPDM family. Then $f$ has no wandering interval.
\end{proposition}

\begin{claim}\label{inj semoconj}
      If $(a,b) \notin \overline{\mathcal{H}}$ and $\phi_{a,b}$ is constant on some interval $J$ then $\bigcup_{n\geq0} f_{a,b}^n(J)=\mathbb{T}$ and $\phi_{a,b}$ is constant on whole $\mathbb{T}$.
\end{claim}
\begin{proof}{(Compare with the proof of Theorem 5.1 of \cite{MMS})}
    Let $\bigcup_{n\geq0} f_{a,b}^n(J)$ be a proper subset of  $\mathbb{T}$.  Let $U$ be a connected component of $\bigcup_{n\geq0} f_{a,b}^n(J)$, then $U$ is also proper subset of $\mathbb{T}$. Since there are no wandering intervals, there is $k \in \mathbb{N}$ so that $U\cap f_{a,b}^k(U) \neq \emptyset$. This means $f_{a,b}^k(U) \subseteq U$. Say $f_{a,b}^k(U) \subset U$ is a proper subset and since $f^k_{a,b}$ is piecewise linear this means $f_{a,b}$ has an attracting periodic point in $U$, hence contradiction. So $f_{a,b}^k(U)=U$. Now since $f_{a,b}$ is increasing both boundary points are fixed. If one of them is attracting then we are done. If both of them are repelling then there is one attracting fixed point of $f_{a,b}^k$ in the interior of $U$ as $f_{a,b}$ is piecewise linear. This is also contradiction.\\
\indent    Now $\phi_{a,b}$ is constant on any component of $\bigcup_{n\geq0} f_{a,b}^n(J)$ since due to the semiconjugacy $\phi_{a,b}$ is constant on $f_{a,b}^n(J)$ for all $n \in \mathbb{N}$. This completes the second part of the proof.
\end{proof}

\begin{theorem}\label{conjugate doubling}
    If $(a,b) \notin \overline{\mathcal{H}}$ then $f_{a,b}$ is conjugate to the doubling map by $\phi_{a,b}$.
    
\end{theorem}
\begin{proof}
    Let $(a,b) \notin \overline{\mathcal{H}}$,  then the semiconjugacy $\phi_{a,b}$ is injective by Claim \ref{inj semoconj}. $\phi_{a,b}$ is non decreasing, continuous and degree 1 by Lemmas \ref{Semiconj},\ref{cont semiconjugacy}. So $\phi_{a,b}$ is actually a homeomorphism. 
\end{proof}

\section{Bifurcation}
The nature of bifurcations in PLPDM family is similar to that in the  Double Standard Family. We first state a lemma that we found interesting.

\begin{lemma}
Let $(a,b) \in \mathcal{P}$ where $b$ is a transcendental number. Then $f_{a,b}^n$ has no linear piece of slope $1$ for any $n \in \mathbb{N}$.
    
\end{lemma}
\begin{proof}
Observe that  $F_{a,b}^n$ is a piecewise linear map and slopes of any linear piece of $F^n_{a,b}$ for any $n$, is of the form $B^{+^{i}}B^{-^{j}}$. If $b$ is transcendental then this slope is never 1 as that would imply that $b$ is algebraic.
\end{proof}

We now state our main result of the section.
\begin{proposition}\label{break bif}
    The attracting periodic point in $I_-$ approaches $1/2$ as $(a,b)$ approaches the left boundary of the hyperbolic component from inside of the hyperbolic component and attracting periodic point in $I_-$ approaches $0$ as $(a,b)$ approaches the right boundary of the hyperbolic component from inside of the hyperbolic component. In particular, $0$ and $\frac{1}{2}$ become  neutral periodic points respectively at the right and  left boundary of any hyperbolic component. 
\end{proposition}
\begin{proof}
    Note that $\mathcal{H}$ is open and $a \mapsto f^i_{a,b}(x)$ is increasing for any $i \in \mathbb{N}$. The rest of the proof is similar to the proof of  Theorem \ref{uniqueness of periodic attractor}.
\end{proof}

The following images are plotted in \textit{Desmos}. These demonstrate the shape of the graph of the fifth iteration when the parameter is in the boundary of a hyperbolic component of period $5$.

\begin{figure}[H]
    \centering
    {{\includegraphics[width=7cm]{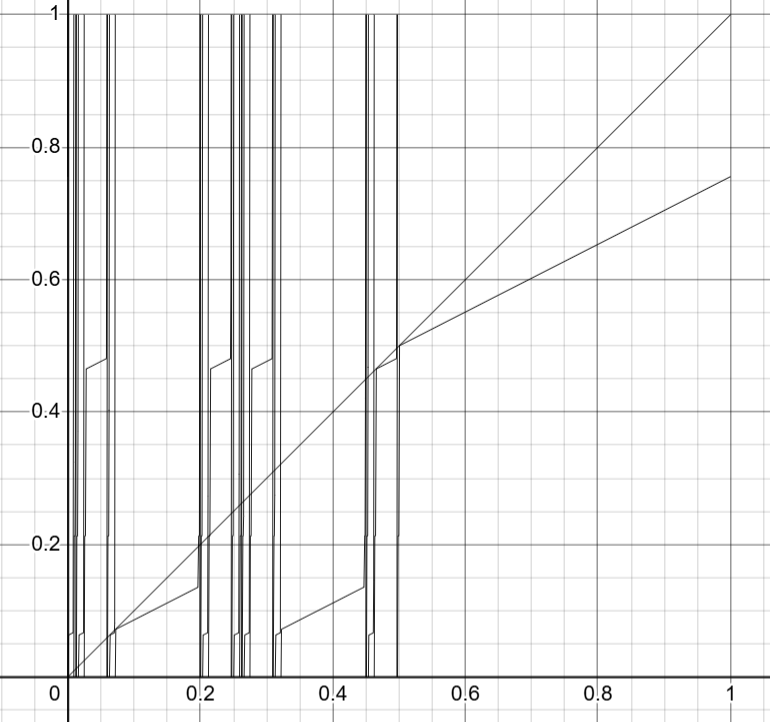} }}%
    \qquad
    {{\includegraphics[width=7cm]{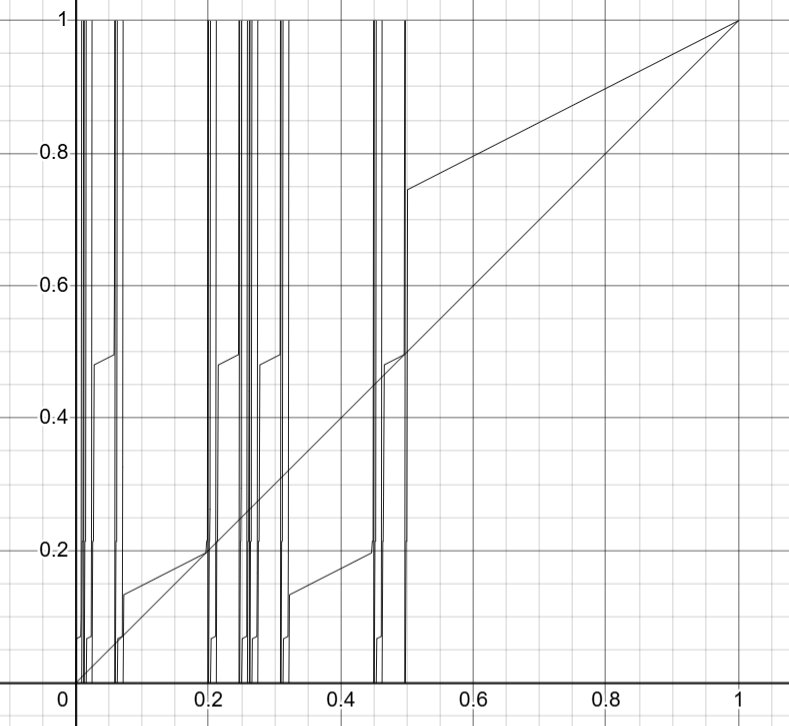} }}
    \caption{ \footnotesize At height $b=0.999$, the horizontal span of period $5$ tongue is from $a\sim0.712957676959782$ (left figure) to $a\sim0.71367603$ (right figure)}%
    \label{fig:example1}%
\end{figure}

\begin{figure}[H]
    \centering
 {{\includegraphics[width=7cm]{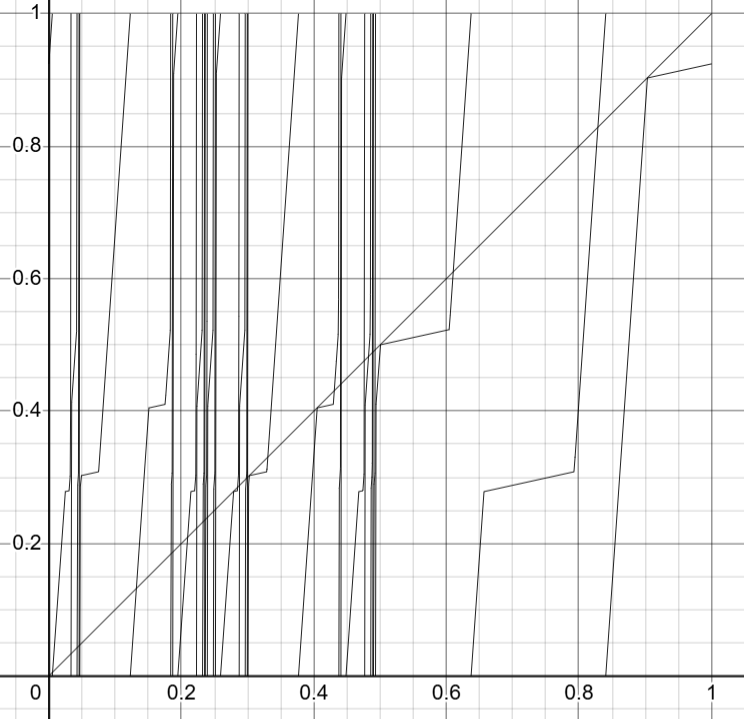} }}%
    \qquad
{{\includegraphics[width=7cm]{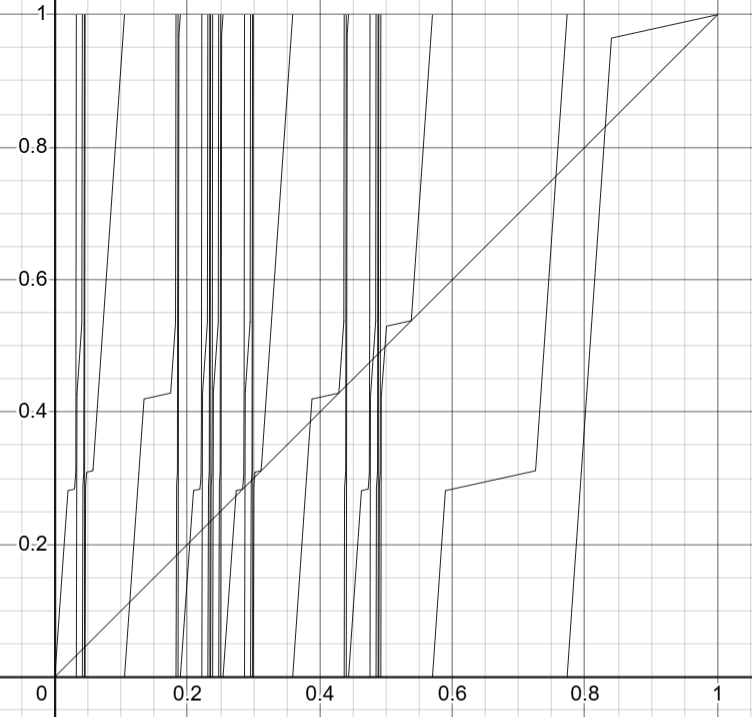} }}%
    \caption{ \footnotesize At height $b=0.96998$, the horizontal span of period $5$ eye is from $ a\sim0.79329$ (left figure) to $a\sim0.79631081199$ (right figure)}%
    \label{fig:example2}%
\end{figure}

Following figures show the behavior at boundary if we zoom in. In the left diagram of Figure \ref{fig:example zoom1}, we can see that  $1/2$ is a periodic point of period 5 at left boundary of a period 5 tongue. It is repelling from left and attracting from right. Because of the lack of resolution the line looks almost vertical in the left of 1/2. The right figure is zoom at 1/2 at the left boundary of a period 5 eye.

\begin{figure}[H]
    \centering
    {{\includegraphics[width=7cm]{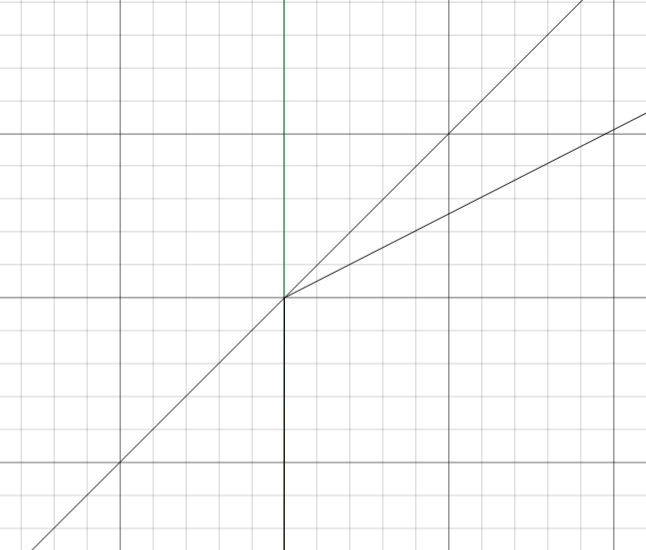} }}%
    \qquad
    {{\includegraphics[width=7cm]{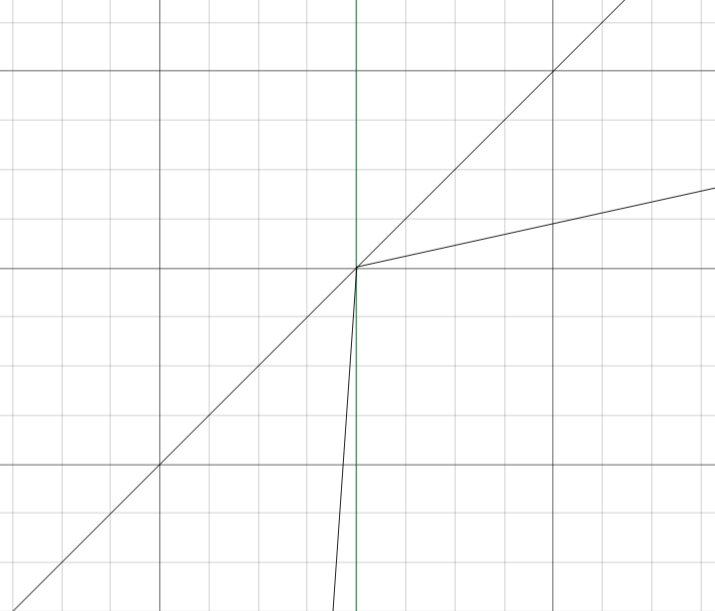} }}
    \caption{ \footnotesize Left figure shows 1/2 as a periodic point of period 5  at parameter  $a=0.712957676959782, b=0.999$ for left boundary of period 5 tongue. The right figure shows that 1/2 is a periodic point of period 5  at parameter   $ a=0.79329,  b=0.96998$ for left boundary of period 5 eye.}%
    \label{fig:example zoom1}%
\end{figure}

The next figure is a zoom in at the right boundary of a period 5 tongue.

\begin{figure}[H]
    \centering
    {{\includegraphics[width=7cm]{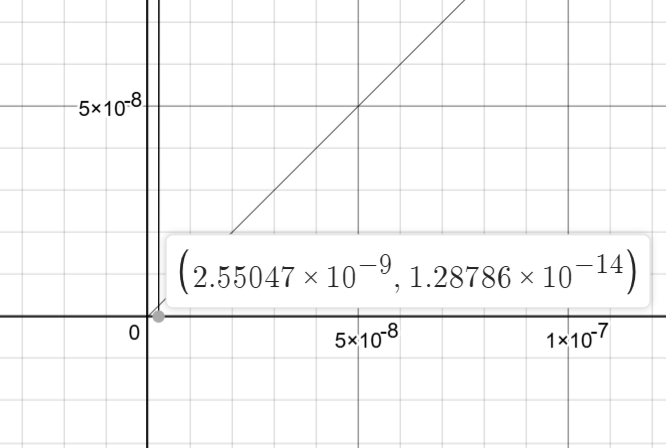} }}%
    \qquad
    {{\includegraphics[width=7cm]{ 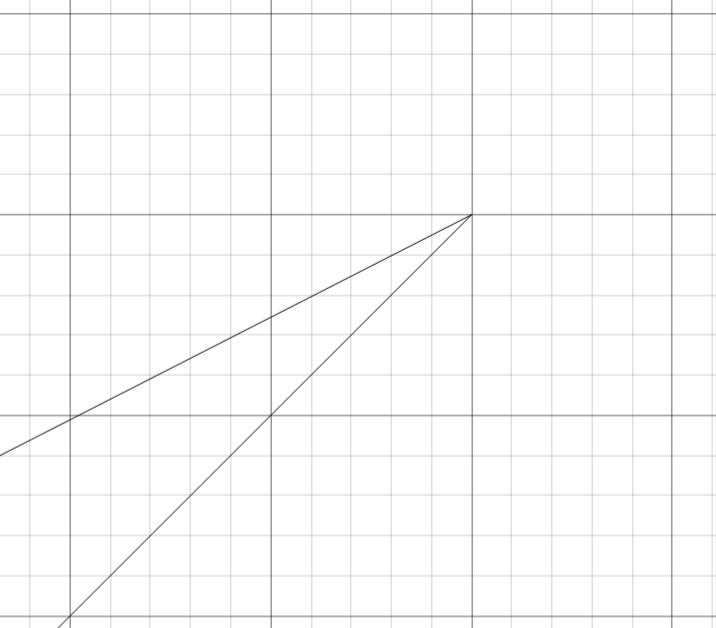} }}
    \caption{ \footnotesize This figure shows $0\sim1$ as a periodic point of period 5 at the right boundary of tongue at parameter  $a=0.71367603, b=0.999$. The left figure shows that $0\sim1$ is repelling from the right and the right figure shows that $0\sim1$ is attracting from left. Because of very high slope the point $(2.555047.10^{-9},1.28786 \times 10^{-14})$ is an approximation of $(0,0)$.} %
    \label{fig:example zoom2}%
\end{figure}

Following figure is a zoom in at the right boundary of a period 5 eye.

\begin{figure}[H]
    \centering
    {{\includegraphics[width=7cm]{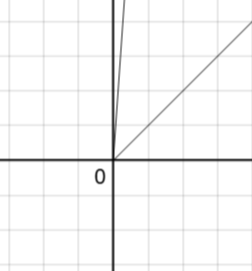} }}%
    \qquad
    {{\includegraphics[width=7cm]{ 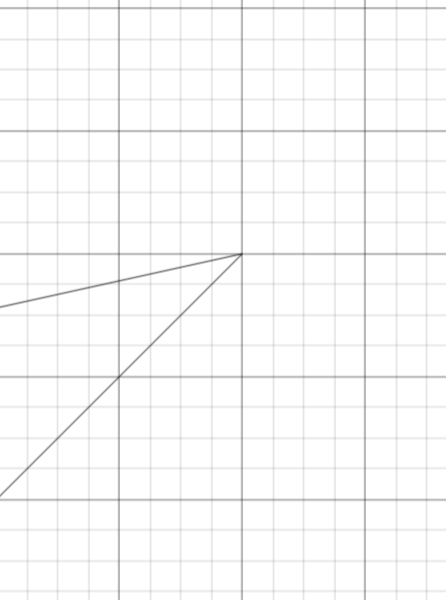} }}
    \caption{ \footnotesize This figure shows $0\sim1$ as a periodic point of period 5 at the right boundary of eye at parameter  $ b=0.96998$. The left figure shows that $0\sim1$ is repelling from the right and the right figure shows that $0\sim1$ is attracting from left.}
    \label{fig:example zoom2}%
\end{figure}

Next picture illustrates that for parameter $a = 0.5760, b = 0.7913$ there is an attracting  cycle  of period 5 with 3 points of the attracting cycle in $I_-$.

\begin{figure}[H]
    \centering
    \includegraphics[scale=.5]{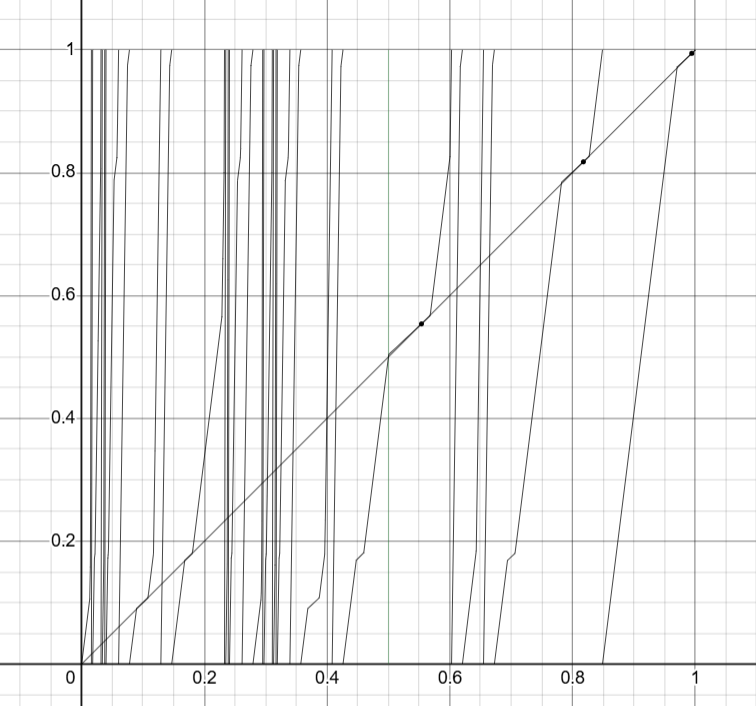}
    \caption{\small Parameter with period 5 attracting cycle with 3 points in attracting cycle in $I_-$.  }
    
\end{figure}
 Here is picture of parameters with period 5 attracting cycle and 3 points in attracting cycle in $I_-$.
 \begin{figure}[H]
    \centering
    \includegraphics[scale=.3]{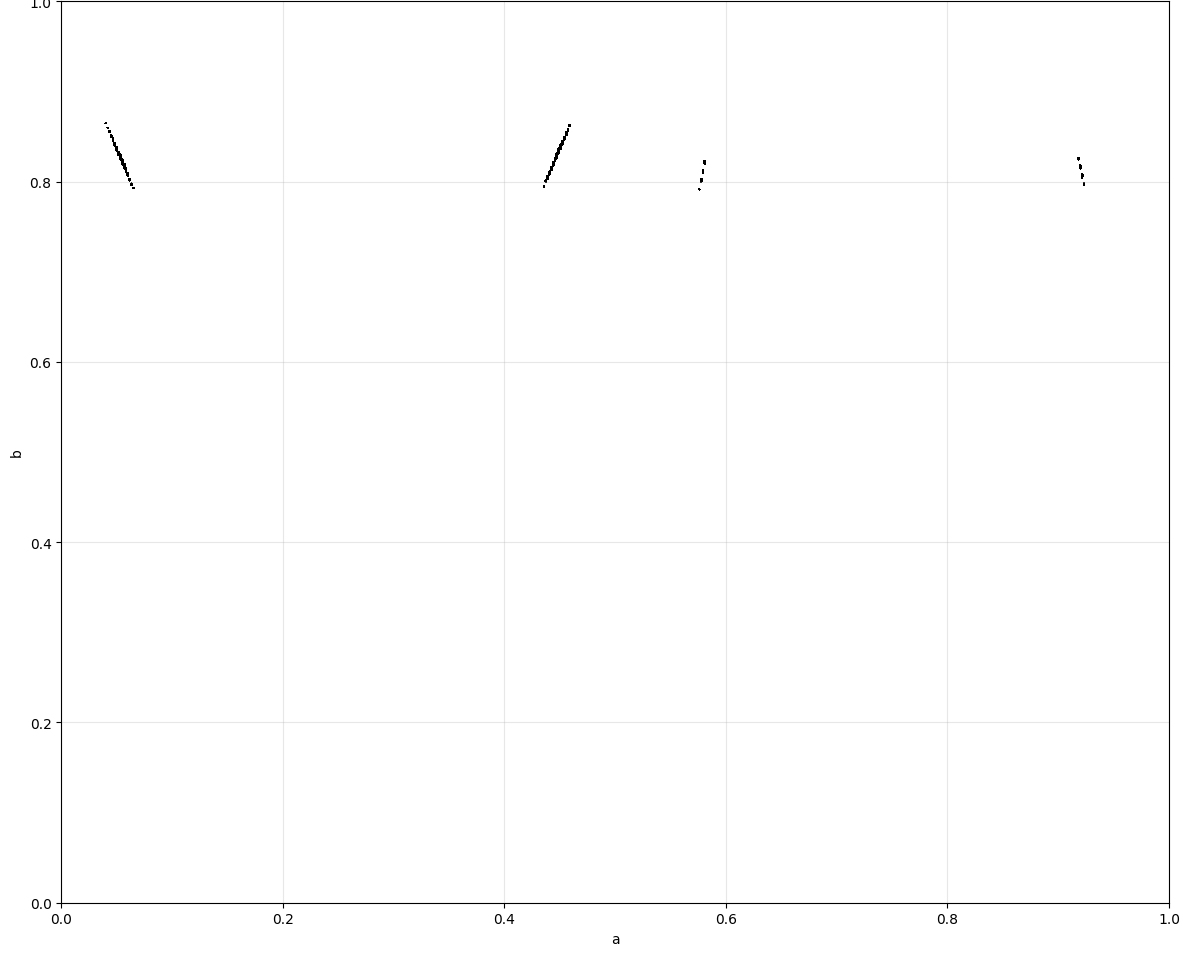}
    \caption{\small parameter with period 5 attracting cycle with 3 points in attracting cycle in $I_-$  }
\end{figure}

\section{Concluding Remarks}
Our findings provide a basic picture of the parameter space of the PLPDM family. Lemma \ref{no only minus} shows that there is no attracting cycle with an itinerary consisting only of $-$.  Our experiments indicate that any other itinerary should be realized by some parameter. We pose it as a question here. 
\begin{qn}
    For $ n \geq 4$,  let $\sigma=(\sigma_0,...,\sigma_{n-1}) $ with $\sigma_i \in \{+,-\}$ be any itinerary with at least one $j \in \{0,...,n-1\}$  such that $\sigma_j=+$. Then does there exist $(a,b) \in \mathcal{H}$ so that $f_{a,b}$ has an attracting cycle with itinerary $\sigma$?
\end{qn}
We believe the answer to this question to be affirmative. This will tell us that all \textit{admissible `post critical' combinatorics} is realized, a situation similar to case of quadratic polynomials.   
\vspace{.1cm}

In \cite{BBM} authors provide a dynamically natural uniformization of the tongues leading to the fact that the Hausdorff dimension of the maximal chaotic set varies real analytically inside the tongues in the DSM family. The maximal set of Devaney chaos is the  complement of the basin of attraction in case of the DSM family. For the PLPDM family, we can similarly define, $C_{a,b}= \mathbb{R/Z}\setminus \mathcal{B}_{a,b}$ where $\mathcal{B}_{a,b}$ is the basin of attraction for $(a,b) \in \mathcal{H}$.
\begin{lemma}
    The set $C_{a,b} \subset \mathbb{R/Z}$ has Lebesgue measure 0, for $(a,b) \in \mathcal{H}$. 
\end{lemma}
\begin{proof}
    Observe that $C_{a,b}$
is completely invariant and compact with the points of non differentiability  $0,1/2 \notin C_{a,b}$. So we can then extend $f_{a,b}|_{C_{a,b}}$ to a $C^2$ map, $\tilde{f}_{a,b}: \mathbb{R/Z} \to \mathbb{R/Z}$ so that $|\tilde{f}_{a,b}'(x)|>\lambda>1$ for $x \in C_{a,b}$. It is clear that $C_{a,b}$ is a hyperbolic, invariant set of $\tilde{f}_{a,b}$. Hence, it follows from  theorem 2.6  of chapter III from \cite{DS}, that  $C_{a,b}$ has either Lebesgue measure 0 or 1. Since $\mathcal{B}_{a,b}$ is open it has  positive Lebesgue measure and the result follows.
\end{proof}
We now ask the following question; $\dim_{H}$ will denote the Hausdorff dimension. 
\begin{qn}
    Is it possible to give a dynamically natural uniformization of a component of $\mathcal{H}$ (this will also establish simple connectedness)? Is $(a,b) \mapsto \dim_{H}(C_{a,b})$ a real analytic map, restricted to a component of $\mathcal{H}$? If not, then what can be said about the regularity of this map?
\end{qn}


\noindent {\bf Acknowledgment:} The research of first author was funded by UGC [NTA Ref. No. 201610319430], Govt. of India. The
second author was supported partly by the Department of Science and Technology (DST), Govt. of India,
under the Scheme DST FIST [File No. SR/FST/MS-I/2019/41] and NBHM travel grant 0207/15(3)/2023/R\&D-II/11964. The authors wish to thank Alexandre Dezotti for various inputs in framing the main research questions and the initial computer experiments. In fact the second author started this project with Dezotti in summer 2019 during his visit to the Imperial College London. The second author wishes to thank Davoud Cheraghi for arranging the visit and the hospitality at the Imperial College London. The authors wish to thank Xavier Buff who suggested to look at the orbits of the break points for showing that there are at most one attracting cycle possible for each map in this family.  The authors are also thankful to Subhamoy Mondal for many helpful discussions.

\end{document}